\documentclass[12pt]{article}
\usepackage{epsfig}
\usepackage{amsfonts}
\usepackage{amssymb}
\usepackage{amstext}
\usepackage{amsmath}
\usepackage{xspace}
\usepackage{theorem}
\usepackage{graphicx}
\usepackage{array}
\usepackage{lineno}
\makeatletter
\newenvironment{proof}{{\bf Proof:  }}{\hfill\rule{2mm}{2mm}}

\newcommand{\junk}[1]{}
\newcommand{\Var}{\mathrm{Var}}
\newtheorem{theorem}{Theorem}

\newtheorem{lemma}[theorem]{Lemma}
\newtheorem{conjecture}{Conjecture}

\newtheorem{corollary}[theorem]{Corollary}

\title{Comparing Eigenvector and Degree Dispersion with the Principal Ratio of a Graph}
\author{Gregory J. Clark\\
\small Sa\"id Business School\\[-0.8ex]
\small University of Oxford\\
}
\begin{document}
\date{} 
\maketitle
\begin{abstract}
The principal ratio of a graph is the ratio of the greatest and least entry of its principal eigenvector. Since the principal ratio compares the extreme values of the principal eigenvector it is sensitive to outliers.  This can be problematic for graphs (networks) drawn from empirical data.  To account for this we consider the dispersion of the principal eigenvector (and degree vector). More precisely, we consider the coefficient of variation of the aforementioned vectors, that is, the ratio of the vector's standard deviation and mean.  We show how both of these statistics are bounded above by the same function of the principal ratio.  Further this bound is sharp for regular graphs.  The goal of this paper is to show that the coefficient of variation of the principal eigenvector (and degree vector) can converge or diverge to the principal ratio in the limit.  In doing so we find an example of a graph family (the complete split graph) whose principal ratio converges to the golden ratio. We conclude with conjectures concerning extremal graphs of the aforementioned statistics and interesting properties of the complete split graph.     
\end{abstract}

\section{Introduction}


There are several measures of graph irregularity which are used to determine how `close' a given graph is to being regular.  One such measure, the \textit{principal ratio}, is the ratio of the greatest and least entry of the principal eigenvector of a graph $G$.  Throughout we assume that graphs are connected, simple, and undirected. We reserve  $(\lambda, x)$ to denote the principal eigenpair of the adjacency matrix of $G$. The principal ratio, then, is denoted $\gamma(G) = x_{\text{max}}/x_{\text{min}}$.  A limitation of this statistic is its sensitivity to the position (and to an extent the degree) of individual vertices in the graph.  For example, its applicability to empirical networks where the principal eigenvector is highly dispersed (e.g., scale-free networks \cite{Bollobas2003, Chung3, Pastor}). This motivates the study of the underlying vector's dispersion.  

Given a vector $x \in \mathbb{R}^n_+$ let $\mu$ denote its mean and $\sigma$ its standard deviation.  The coefficient of variation of $x$ is the ratio $\sigma/\mu$.  For convenience we will consider the square of the coefficient of variation which we denote 
\[
c_x := \left(\frac{\sigma}{\mu}\right)^2.
\]

The coefficient of variation and its square appear frequently across multiple disciplines: economic inequality \cite{Atkinson}, reliability of measures \cite{Shechtman}, and efficiency of experimental designs \cite{Grubbs}.  Interestingly, the inverse of the coefficient of variation, $\mu/\sigma$, is the signal-to-noise ratio (SNR).  The SNR is used informally to measure the ratio of useful to irrelevant information in a conversation \cite{Gragido}.  It has further found use in information theory \cite{Shannon}, optics \cite{Rose}, internet topology \cite{Kumar}, medical imaging \cite{Edelstein}, and nuclear engineering \cite{Hoult}.

Studying the distribution of vectors associated to graphs is not new.  Notably the variance of the degree vector was an early irregularity measure \cite{Bell}.  As such, we consider the coefficient of variation of the degree vector in addition to the principal eigenvector.  We reserve $d$ to denote the degree vector of a graph $G$.  We denote the square of the coefficient of variation of the principal eigenvector as $c_e$ and the degree vector as $c_d$. Studying the distribution of these two vectors in tandem will further aid in motivating their use in practice.  As we will see there are instances where $c_e$ and $c_d$ diverge.


The goal of this paper is to determine how much the principal ratio and the coefficient of variation of the principal eigenvector (and degree vector) can differ.  We begin by showing that $c_d$ and $c_e$ are both bounded above by the same function of principal ratio (namely, $\gamma^2-1$).  We then consider various graph families and compute the limit of the aforementioned statistics of said graphs.  In particular, we present graph families for which the limit is 0, non-zero, or diverges to infinity.  We summarize our results in Table \ref{T:Results} and note that all six cases for the limit of the principal ratio and the coefficient of variation (for the principal eigenvector or degree vector, respectively) are achieved.  We provide conjectures for which graphs are extremal for the greatest eigenvector and degree dispersion.  In doing so we introduce a statistic 
\[
\Gamma(G) = \frac{c_d(G) - c_e(G)}{\gamma^2(G)}
\]
and provide a conjecture for which graph achieves the lower bound.  We show that the limit of the principal ratio of the complete split graph ratio tends to $\varphi$, the golden ratio.  We conclude by showing that the complete spit graph also has an interesting empirical property: its local and global clustering diverge.  This is (to the best of our knowledge) the third such example of a graph family to do so.  



\begin{table}
\caption{The limit of $\gamma^2-1$, $c_e$, and $c_d$ for various graph families including the complete graph with an edge removed; a particular complete triparite graph, a complete split graph (i.e., $S(n,m) = K_{n+m} - K_{m})$; the complete graph with a pendant edge, a kite whose head is an $r$-regular graph; the star; and the Cartesian product of a graph with itself.}
\begin{center}
\begin{tabular}{ |c|c|c|c| } 
 \hline
 $G_n$ & $\lim \gamma^2-1$ & $\lim c_e$ & $\lim c_d$ \\  \hline \hline
 $K_{n}-K_2$ & $0$ & $0$ & $0$ \\ \hline
 $K_{1,n,n}$ &  $3$ & $0$ & $0$ \\  \hline
 $S(n,kn)$  & $\left(\frac{\sqrt{4k+1}+1}{2}\right)^2-1$ & $ \frac{k\left(\frac{\sqrt{4k+1}-1}{k}-2\right)^2}{(\sqrt{4k+1}+1)^2}$ & $\frac{k^3}{(2k+1)^2}$ \\ \hline

 $P_2K_{n-1}$ & $\infty$ & $0$ &$0$\\\hline
 $P_nG_n^r$ & $\infty$ & $-$ & $\left(\frac{r-2}{r+2}\right)^2$ \\\hline
 $K_{1,n}$ & $\infty$ & 1 & $\infty$ \\\hline
 $G^{\square n}$ & $\infty$ & $\infty$ & 0  \\\hline
\end{tabular}
\label{T:Results}
\end{center}
\end{table}

\section{Background}

Throughout we assume that graphs are connected, simple, and undirected.  We consider the distribution of a vector so that the indices are often irrelevant.  To better facilitate the discussion we will abuse notation and write
\[
x = ((x_i, m_i))_{i \geq 1} \text{ for } x \in \mathbb{R}^n
\]
to mean that $x_i$ occurs multiplicity $m_i$ in $x$.  Let $\mu$ and $\sigma$ denote the mean and standard deviation of $x$, then the \emph{square of the coefficient of variation} is $c_x = (\sigma/\,u)^2$.  With a slight abuse of notation we write $c_e(G)$ and $c_d(G)$ to be the square of the coefficient of variation for the principal eigenvector and degree vector of $G$, respectively.  When the context is clear we will simply write $c_e$ (respectively, $c_d$).  

In some cases we consider the difference of two graphs.  Let $H \subseteq G$ be two graphs on the same vertex set.  We write $G-H$ to be the graph formed by removing the edges of $H$ from $G$.  That is $E(G - H) = \{e \in G: e \notin H\}$ and $V(G-H) = V(G)$.

We begin with the following fact about the coefficient of variation of a positive vector.  

\begin{lemma}
\label{L:fact}
Let $x \in \mathbb{R}^n_{+}$ then
\begin{equation}
\label{E:formula}
c_x = n\frac{||x||_2^2}{||x||_1^2} - 1.
\end{equation}
Moreover, 
\[
c_x \leq \left(\frac{x_{\text{max}}}{x_{\text{min}}}\right)^2-1.
\]
\end{lemma}

\begin{proof}
Let $X$ to be the random variable which takes on $x_i$ with probability $1/n$.  Intuitively, $X$ draws a coordinate-value of $x$ uniformly at random. We have
\[
\mathbb{E}[X] = \frac{||x||_1}{n} \text{ and } \mathbb{E}[X^2] = \frac{||x||^2_2}{n}
\]
so that
\[
 ||x||_1^2 = n^2\mathbb{E}[X]^2 \text{ and } ||x||_2^2 = n \mathbb{E}[X^2] .
\]

It follows that
\[
\frac{||x||_2^2}{||x||_1^2} = \frac{\mathbb{E}[X^2]}{n\mathbb{E}[X]^2}.
\]
Whence $\Var(x) = \mathbb{E}[X^2] - \mathbb{E}[X]^2$ we have
\[
\frac{||x||_2^2}{||x||_1^2}  = \frac{\Var(X) + \mathbb{E}[X]^2}{n\mathbb{E}[X]^2}.
\]
Indeed 
\[
n\left(\frac{||x||_2^2}{||x||_1^2}\right) - 1 = \frac{\Var(X) + \mathbb{E}[X]^2}{\mathbb{E}[X]^2} - 1 = \frac{\Var(X)}{\mathbb{E}[X]^2} = c_v(x).
\]

Now assume that $||x||_2^2 = 1$.  By an averaging argument we have that $x_\text{max} \geq n^{-1/2}$ so that
\[
||x||_1^2 \geq n^2x_\text{min}^2 = x_\text{max}^2n^2\left(\frac{x_{\text{min}}}{x_{\text{max}}}\right)^2 \geq n \left(\frac{x_{\text{min}}}{x_{\text{max}}}\right)^2.
\]
The desired inequality follows by substitution into Equation \ref{E:formula}.
\end{proof}

This immediately implies the following.

\begin{lemma}
\label{L:eigen}
$c_e(G) \leq \gamma^2(G)-1$.
\end{lemma}

Consider $c_d(G)$ and note that Lemma \ref{L:fact} yields $c_d \leq (\Delta/\delta)^2-1$.  Since $c_d(G)$ is a function of the degree vector and $\gamma(G)$ is a function of the principal eigenvector it is natural to think that they are incomparable, \textit{a priori}.  Surprisingly $c_d(G)$ is similarly bounded above and the bound is sharp for regular graphs. 

\begin{lemma}
\label{L:degree}
$c_d(G) \leq \gamma^2(G)-1$.
\end{lemma}

\begin{proof}
We have 
\[
||d||_1^2 = (2|E|)^2 \text{ and } ||d||_2^2 = \sum_{uv \in E(G)} \deg(u) + \deg(v) \leq 2|E|\Delta
\]
so that by Equation \ref{E:formula},
\[
c_d \leq n\left(\frac{2|E|\Delta}{4|E|^2}\right)-1 = \frac{n\Delta}{2|E|}-1 = \frac{\Delta}{\bar d} - 1 \leq \frac{\Delta}{\delta}-1.
\]
From \cite{Cioaba, ZhangXD} we have $\sqrt{\Delta/\delta} \leq \gamma$ yielding $c_d \leq \gamma^2-1$ as desired. 
\end{proof}

The coefficient of variation is invariant under normalization.  When considering $c_e$ it can be useful to specify a given normalization.  In some cases we will assume that $||x||_2^2 =1$ or we will assume $x_{\text{max}}$ (vis-a-vis $x_\text{min}$) is 1.

We have further shown the following.

\begin{lemma}
\label{L:Ce_bound}
$c_e \leq \frac{n}{\lambda+1} - 1$.
\end{lemma}

\begin{proof}
Let $(\lambda, x)$ be the principal eigenpair of $G$ so that $||x||_2^2 = 1$.  Appealing to the Rayleigh quotient definition of the principal eigenvector we have
\[
\lambda = 2\sum_{uv \in E} x_ux_v.
\]
Note that
\[
||x||_1^2 = \sum_{u,v \in V}x_ux_v \geq  \left(2\sum_{uv \in E} x_ux_v \right)+ \sum_{u \in V} x_u^2 +   = \lambda + 1.
\]
By Equation \ref{E:formula} we have
\[
c_e(G) = n\frac{||x||_2^2}{||x||_1^2} - 1 \leq \frac{n}{\lambda +1} - 1.
\] 
\end{proof}

The following was a desideratum of Lemma \ref{L:degree}.
\begin{lemma}
\label{L:Cd_bound}
$c_d(G) \leq (\Delta/\overline{d})-1$ where $\overline{d}$ is the average degree.
\end{lemma}

\section{Dispersion of Various Graph Families}

We now consider the graph families listed in Table \ref{T:Results} in descending order. 

\subsection{The complete graph with an edge removed}
In this section we consider $K_n-K_2$, the complete graph with an edge removed. 
\begin{theorem}
Let $G_n = K_n-\{1,2\}$ with $n > 2$.  Then 
\[
\lim_{n \to \infty} \gamma(G_n)= 0,\; \lim_{n \to \infty} c_e(G_n) = 0, \; \text{ and } \lim_{n \to \infty} c_d(G_n) = 0.
\]
\end{theorem}

\begin{proof}
Let $G_n = K_n-\{1,2\}$ with $n > 2$ and let $(\lambda, x)$ be its principal eigenpair.

By symmetry we have $a:= x_1, x_2$ and $b:= x_i$ for $i > 2$ so that the eigenequations satisfy
\begin{align*}
    \lambda a &= (n-3)a + 2b\\
    \lambda b &= (n-2)a.
\end{align*}

Setting $b = 1$ yields 
\[
(n-2)a^2 - (n-3)a - 2 = 0
\]
so that 
\[
a = \frac{(n-3) + \sqrt{n^2+2n-7}}{2(n-2)}.
\]
Note that
\[
x = \left(\left(\frac{(n-3) + \sqrt{n^2+2n-7}}{2(n-2)},2\right), (1,n-2)\right)
\]
We have that
\[
\lim_n\gamma(K_n-K_2) = \lim_n a = 1.
\]
As $\lambda = (n-2)a$ we appeal to Lemma \ref{L:Ce_bound} to conclude
\[
\lim_{n \to \infty} c_e \leq \lim_{n \to \infty} \frac{n}{\lambda+1} -1 = \lim_{n \to \infty} \frac{n}{n}-1 = 0.
\]
Finally, whence $\Delta = n-1$ and $\delta = n-2$ we have from Lemma \ref{L:Cd_bound}
\[
\lim_{n \to \infty} c_d \leq \lim_{n \to \infty} \frac{\Delta}{\delta} - 1 = \lim_{n \to \infty} \frac{n-1}{n-2} -1 = 0.
\]
\end{proof}

\subsection{The complete tripartite graph}

Let $K_{a,b,c}$ be the complete tripartite graph with clouds of size $a,b,c$.

\begin{theorem}
We have
\[
\lim \gamma(K_{1,n,n}) = 2, \lim c_e(K_{1,n,n}) = 0,\text{ and } \lim c_d(K_{1,n,n}) = 0.
\]
\end{theorem}

\begin{proof}
Consider $K_{1,n,n}$ and let $(\lambda, x)$ be its principal eigenpair.  By symmetry we have that $a:= x_1$ and $b:= x_i $ for $i > 1$ so that
\[
x = ((a,1), (b,2n)).
\]
Thus the eigenequations of $K_{1,n,n}$ are
\begin{align*}
    \lambda a &= 2nb \\
    \lambda b &= 1+nb.
\end{align*}
Setting $a = 1$ and solving for $b$ yields 
\[
b = \frac{n + \sqrt{n^2 + 8n}}{4n}
\]
so that
\[
x = \left((1,1), \left(\frac{n + \sqrt{n^2 + 8n}}{4n},2n\right)\right)
\]
We have 
\[
\lim_{n \to \infty} \gamma (K_{1,n,n}) = \lim_{n \to \infty}\frac{4n}{n + \sqrt{n^2 + 8n}} = 2.
\]

We further have 
\[
||x||_2^2 = 1 + 2nb^2 \text{ and } ||x||_1^2 = (1 + 2nb)^2
\]
which yields
\[
\lim_{n \to \infty} c_e(K_{1,n,n}) = \lim_{n \to \infty} (2n+1)\frac{||x||_2^2}{||x||_1^2} - 1 = 2n\left(\frac{n/2}{n^2}\right) -1 = 0.
\]
Moreover,
\[
d = ((2n,1), (n+1,2n))
\]
implying
\[
||d||_2^2 = (2n)^2 + 2n(n+1)^2, ||d||_1^2 = (2n+2n(n+1))^2.
\]
It follows that 
\[
\lim_{n \to \infty} c_d(K_{1,n,n}) = \lim_{n \to \infty} (2n+1)\frac{||d||_2^2}{||d||_1^2}-1 = 2n\left(\frac{2n^3}{4n^4}\right)-1 = 0.
\]
\end{proof}

\subsection{The complete split graph}

\begin{figure}[h]
    \centering
    \includegraphics[width=0.2\textwidth]{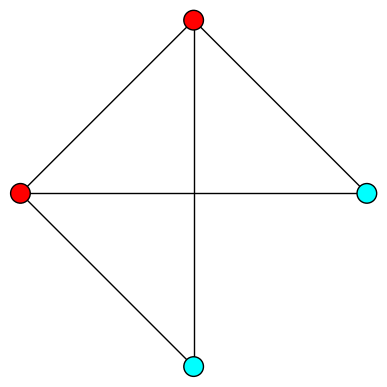}
    \includegraphics[width=0.2\textwidth]{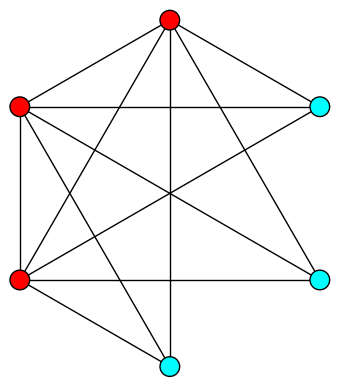}
    \includegraphics[width=0.2\textwidth]{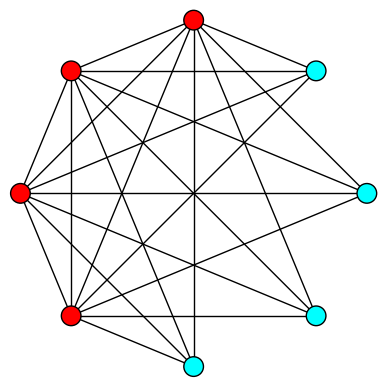}
    \includegraphics[width=0.2\textwidth]{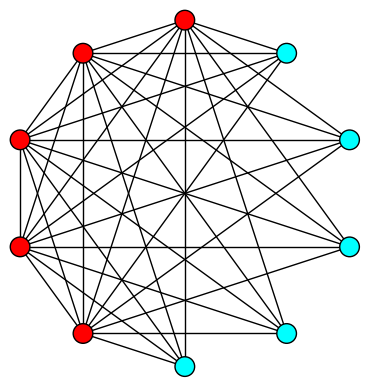}
    \caption{The complete split graph $S(n,n)$ for $2 \leq n \leq 5$.}
    \label{F:Quasi}
\end{figure}

We define the complete split graph to be a clique on $n$ vertices an independent set with $m$ vertices, $S(n,m) = K_{n+m}-K_m$ \cite{Foldes1977, Hammer1981, Tyshkevich1979}. A drawing of $S(n,n)$ is given in Figure \ref{F:Quasi}.  There are several existing graph structures related to $S(n,m)$.  Observe that the complete split graph can be seen as the union of a complete bipartite graph and a clique, $S(n,m) = K_{n,m} \cup K_n$ where multiple edges are ignored.  Note that $S(1,m) = K_{1,m}$ is the star with $m$ rays and $S(2,m)$ is the \textit{agave graph} as described in \cite{Estrada2}.  Further $S(n,m)$ can be viewed as the \emph{multicone} of a complete graph joined with isolated vertices (as in \cite{Wang}) so that $K_n \triangledown \overline{K}_m$ where $\overline{K}_m$ is the edgeless graph on $m$ vertices. 
\begin{lemma}
\label{L:Quasi}
Let $(\lambda, x)$ be the principal eigenpair of $S(n,m)$ then
 \begin{displaymath}
   x_v = \left\{
     \begin{array}{lr}
       1 & : v \leq n\\
       \frac{1 - n + \sqrt{(n-1)^2 + 4nm}}{2m} & : v > n.
     \end{array}
   \right.
\end{displaymath} 
Furthermore, under this normalization $x_\text{max} = 1$.
\end{lemma}

\begin{proof}
We have by symmetry that $a = x_v$ for $v \leq n$ and $b = x_v$ for $v > n$.  The eigenequations are of the form
\begin{align*}
    \lambda a &= (n-1)a + mb \\
    \lambda b &= na.
\end{align*}
Setting $a = 1$ and substituting the first eigenequation into the second yields
\[
mb^2 + (n-1)b - n = 0.
\]
Solving for $b$ gives the desired result. 

By straight forward computation we find that 
\[
\frac{1 - n + \sqrt{(n-1)^2 + 4nm}}{2m} < 1
\]
if and only if $m(m-1) > 0$.  Since $m > 1$ the inequality is always satisfied.
\end{proof}

\begin{theorem}
\label{L:princ_quasi}
Fix $k \geq 1$,
\[
\lim_{n \to \infty} \gamma(S(n,k\cdot n)) = \frac{\sqrt{4k+1}+1}{2}.
\]
\end{theorem}

\begin{proof}
From Lemma \ref{L:Quasi} we have that the principal ratio of $S(n,m)$ is 
\[
\gamma(S(n,kn)) = \frac{1}{b} = \frac{2kn}{1-n + \sqrt{(n-1)^2 + 4kn^2}}.
\]
We have then 
\[
\lim_{n \to \infty} \gamma(S(n,kn)) = \lim_{n \to \infty} \frac{2kn}{(\sqrt{4k+1} - 1)n} =  \frac{\sqrt{4k+1}+1}{2}.
\]
\end{proof}

\begin{corollary}
\label{C:golden}
We have $\lim_{n \to \infty} \gamma(S(n,n)) = \varphi$ is the golden ratio.
\end{corollary}

\begin{theorem}
Fix $k \geq 1$, 
\[
\lim_{n \to \infty} c_e(S(n,kn)) = \frac{k\left(\frac{\sqrt{4k+1}-1}{k}-2\right)^2}{(\sqrt{4k+1}+1)^2}.
\]
\end{theorem}

\begin{proof}
Let $(\lambda, x)$ be the principal eigenpair of $S(n,m)$ as in Lemma \ref{L:Quasi}.  Clearly
\[
\mu = \frac{n + mb}{n+m}
\]
and 
\[
\sigma^2 = \frac{(b-1)^2(nm^2+n^2m)}{(n+m)^3}.
\]
We have then 
\[
c_e = \frac{\sigma^2}{\mu^2} = \frac{(b-1)^2(nm^2+n^2m)}{(n+m)(n+mb)^2}.
\]
Setting $m = kn$ yields
\[
c_e = \frac{(b-1)^2k}{(kb+1)^2}.
\]
From the proof of Theorem \ref{L:princ_quasi} we have

\[
\lim_{n \to \infty} b =\lim_{n \to \infty} \gamma(S(n,kn))^{-1} = \frac{2}{\sqrt{4k+1}+1}
\]
so that 
\[
\lim_{n \to \infty} c_e(S(n,kn)) = \lim_{n \to \infty} \frac{(b-1)^2k}{(kb+1)^2} = \frac{k\left(\frac{\sqrt{4k+1}-1}{k}-2\right)^2}{(\sqrt{4k+1}+1)^2}.
\]
\end{proof}

\begin{theorem}
Fix $k$ then
\[
\lim_{n \to \infty} c_d(S(n,kn)) = \frac{k^3}{(2k+1)^2}.
\]
\end{theorem}

\begin{proof}
Let $d$ be the degree vector of $S(n,m)$ so that 
\[
d = ((n+m-1,n),(n,m)).
\]
We have
\[
\mu = \frac{n^2 + 2nm - n}{n+m}
\]
and 
\[
\sigma^2 = \frac{(n(n+m-1-\mu)^2 + m(n-\mu)^2}{n+m} = \frac{(m-1)^2(nm^2+n^2m)}{(n+m)^3}.
\]
We equate
\[
c_d(S(n,m) = \frac{\sigma^2}{\mu^2} = \frac{(m-1)^2(nm^2 + n^2m)}{(n+m)(n^2+2nm-n)^2}
\]
substituting $m=kn$ yields
\[
c_d(S(n,kn) = \frac{kn^2(kn-1)^2}{(n^2(2k+1)-n)^2}.
\]
It follows that 
\[
\lim_{n \to \infty} c_d(S(n,kn) = \frac{k^3}{(2k+1)^2}.
\]

\end{proof}

\subsection{The kite graph}

The kite graph (aka lollipop graph) $P_mK_s$ is formed by identifying the last vertex of a path on $m$ vertices with a vertex from the complete graph on $s$ vertices.  Examples are given in Figure \ref{F:Kite}.  Kite graphs appear frequently in the literature of the principal ratio \cite{Brightwell, Tait}.  We begin with a bound on the spectral radius of the kite graph and an explicit formula for the principal eigenvector of the kite graph as provided in \cite{Cioaba}.

\begin{figure}[h]
    \centering
    \includegraphics[width=0.2\textwidth]{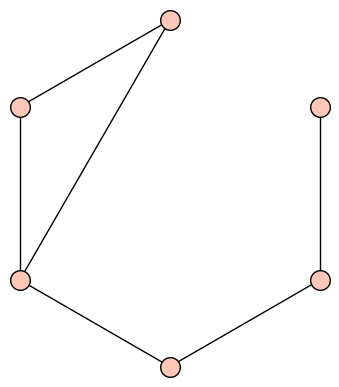}
    \includegraphics[width=0.23\textwidth]{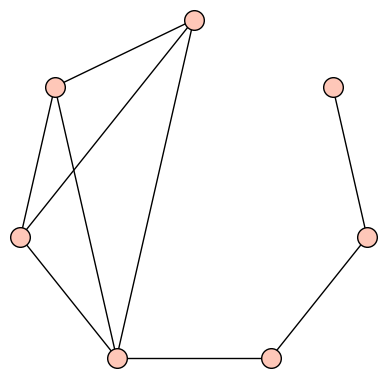}
    \includegraphics[width=0.23\textwidth]{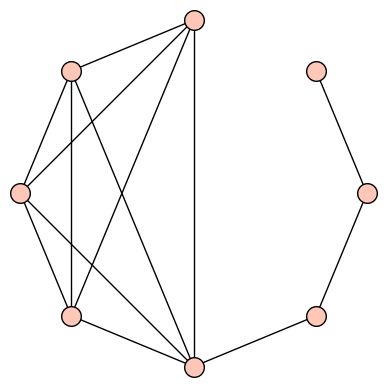}
    \caption{$P_4K_3, P_4K_4$ and $P_4K_5$.}
    \label{F:Kite}
\end{figure}

\begin{lemma}
\label{L:kite_radius}
(\cite{Cioaba}) For $m \geq 2$ and $s \geq 3$, 
\[
s-1 + \frac{1}{s(s-1)} < \lambda(P_m\cdot K_s) < s-1 + \frac{1}{(s-1)^2}.
\]
\end{lemma}
We now show how the entries of the principal eigenvector of $P_mK_s$ can be expressed in terms of the spectral radius.
\begin{lemma}
\label{L:kite_vec}
(\cite{Cioaba}) Let $\lambda$ be the greatest eigenvalue of $P_mK_s$ and let $x$ be its principal eigenvector.  Let $\sigma = (\lambda + \sqrt{\lambda -4})/2$ and $\tau = \sigma^{-1}$ then
\[
x_k = \frac{\sigma^k-\tau^k}{\sigma-\tau}x_1 \text{ for } 1 \leq k \leq m,
\]
and 
\[
x_k = \frac{1}{s-1}\frac{\sigma^{m+1}-\tau^{m+1}}{\sigma-\tau} x_1 \text{ for } m+1 \leq k \leq n.
\]
\end{lemma}

For simplicity we consider the kite whose path is of length 2 (i.e., $K_n$ with a pendant edge).

\begin{theorem}
Consider $P_2K_{n-1}$.  We have
\[
\lim_{n \to \infty} \gamma(P_2K_{n-1}) = \infty, \lim_{n \to \infty} c_e(P_2K_{n-1}) = 0, \lim_{n \to \infty} c_d(P_2K_{n-1}) = 0.
\]
\end{theorem}

\begin{proof}
From Lemma \ref{L:kite_vec} we have
\[
\gamma(P_2K_{n-1}) = \frac{\sigma^2 - \tau^2}{\sigma-\tau} = \sigma + \tau > \sigma. 
\]
Appealing to Lemma \ref{L:kite_radius} we find $\lambda > n-2$ so that $\sigma > (n-2)/2$ hence 
\[
\lim \gamma(P_2K_{n-1}) = \infty.
\]
Similarly, we have from Lemma \ref{L:Ce_bound} that
\[
c_e(P_2K_{n-1}) \leq \frac{n}{\lambda + 1} - 1 \leq \frac{n}{n-1} - 1
\]
which yields $\lim c_e(P_2K_{n-1}) = 0$.  Finally, we have that $\Delta(P_2K_{n-1}) = n-1$ and 
\[
\overline{d} = \frac{n-1 + (n-2)^2 + 1}{n} = \frac{n^2-3n+4}{n}
\]
appealing to Lemma \ref{L:Cd_bound}
\[
\lim c_d(P_2K_{n-1}) \leq \lim \frac{n(n-1)}{n^2-3n+4} - 1 = 0.
\]
\end{proof}

\subsection{A kite whose head is a regular graph}
Let $G_n^r$ be an $r$-regular connected graph on $n$ labeled vertices.  Consider the graph $P_mG_n^r$ which is formed from identifying the $m$-vertex of $P_m$ with the $1$-vertex in $G_n^r$.  Two examples are given in Figure \ref{F:Reg_Kite}.

\begin{figure}[h]
    \centering
    \includegraphics[width=0.2\textwidth]{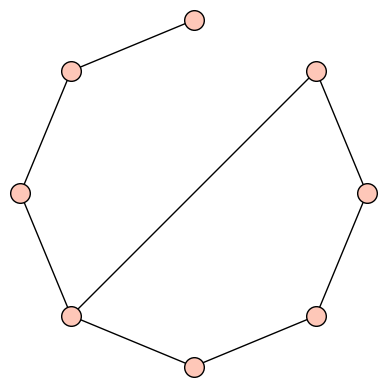}
    \includegraphics[width=0.2\textwidth]{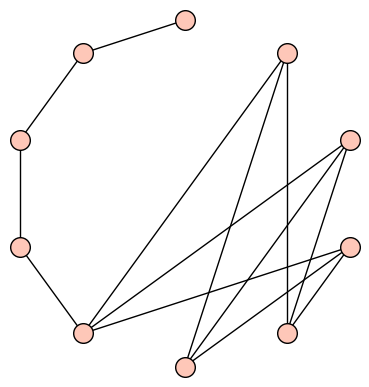}
    \caption{$P_4C_5$ and $P_5K_{3,3}$.}
    \label{F:Reg_Kite}
\end{figure}

\begin{theorem}
\label{T:Reg_Kite}
We have 
\[
\lim_{n \to \infty}c_d(P_nG_n^r) = \left(\frac{r-2}{r+2}\right)^2.
\]
\end{theorem}

\begin{proof}
By construction, 
\[
d = ((1,1), (2,n-2), (r, n-1), (r+1,1))
\]
so that
\[
||d||_2^2 = r^2n + 4n+2r-6 \text{ and } ||d||_1^2 = r^2n^2 + 4rn^2 + 4n^2 - 4rn-8n+4.
\]
We have by Equation \ref{E:formula}
\begin{align*}
\lim_{n \to \infty} c_d(P_nG_n^r) &= \lim_{n \to \infty}\frac{ (2n-1)(r^2n + 4n+2r-6)}{r^2n^2 + 4rn^2 + 4n^2 - 4rn-8n+4} - 1 \\
&= \frac{2(r^2+4)}{r^2+4r+4} -1 = \left(\frac{r-2}{r+2}\right)^2.
\end{align*}
\end{proof}

 We make use of the following Theorem.

\begin{theorem}
\label{T:pen}
\cite{Cioaba} Let $G$ be a connected graph of order $n$ with spectral radius $\lambda > 2$ and principal eigenvector $x$.  Let $d$ be the shortest distance from a vertex on which $x$ is maximum to a vertex on which it is minimum.  Then 
\[
\gamma(G) \leq \frac{\sigma^{d+1}-\tau^{d+1}}{\sigma - \tau}
\]
where $\sigma = \frac{1}{2}(\lambda + \sqrt{\lambda^2 - 4})$ and $\tau = \sigma^{-1}$.  
Equality is attained if and only if $G$ is regular or there is an induced path of length $d > 0$ whose endpoints index $x_{min}$ and $x_{max}$ and the degrees of the endpoints are $1$ and $3$ or more, respectively, while all other vertices of the path have degree 2 in $G$.
\end{theorem}

\begin{theorem}
Fix $r \geq 4$.  Then
\[
\lim_{n \to \infty}\gamma(P_nG_n^r) = \infty.
\]
\end{theorem}

\begin{proof}
Let $(\lambda, x)$ be the principal eigenpair of $P_nG_n^r$.  Since $G_n^r$ is an induced subgraph of $P_nG_n^r$ we have that $\lambda \geq r-1$.  Clearly $x_{\text{min}} = x_1$ and $x_\text{max} = x_n$.  As $P_nG_n^r$ satisfies the conditions of Theorem \ref{T:pen} we have that 
\[
\gamma(P_nG_n^r) = \frac{\sigma^{n+1}-\tau^{n+1}}{\sigma - \tau} \geq \sigma^n.
\]
Whence $r \geq 4$ we have that $\sigma > 1$ so that
\[
\lim_{n \to \infty} \gamma(P_nG_n^r) = \infty.
\]
\end{proof}

\subsection{The star}

We now consider the star graph $K_{1,n}$.

\begin{theorem}
\label{L:star_limits}
We have
\[
\lim_{n \to \infty} \gamma(K_{1,n})^2 -1 = \infty,\; \lim_{n \to \infty} c_e(K_{1,n}) = 1, \; \text{ and } \lim_{n \to \infty} c_d(K_{1,n}) = \infty.
\]
\end{theorem}

\begin{proof}
Let $K_{1,n}$ denote the star with $n$ rays.  It is well known that
\[
x = ((1,1), (n^{-1/2},n)) \text{ and } d = ((n,1),(1,n)).
\]
Clearly
\[
\lim_{n \to \infty} \gamma(K_{1,n})^2-1 = \lim_{n \to \infty} n-1 = \infty.
\]
We have that $||d||_1^2 = 4n^2$ and $||d||_2^2 = n^2 + n$ so that
\[
c_d(K_{1,n}) = (n+1) \frac{||d||_2^2}{||d||_1^2} - 1 = \frac{n^3+2n^2+n}{4n^2}-1
\]
and $\lim_nc_d(K_{1,n}) = \infty$.

Moreover,
\[
||x||_1^2 = (1 + n^{-1/2})^2 \text{ and } ||x||_2^2 = 2.
\]
As such 
\[
c_e(K_{1,n}) = (n+1)\frac{||x||_2^2}{||x||_1^2} - 1 = \frac{2(n+1)}{(n^{1/2}+1)^2}-1
\]
so that $\lim_{n \to \infty} c_e(K_{1,n}) = 1$.
\end{proof}

\subsection{Cartesian powers of a graph}


Given two graphs $A$ and $B$, the Cartesian product $A\square B$ is the graph on $|V(A)||V(B)|$ vertices where $V(A \square B) = \{(a,b): a \in V(A), b \in V(B)\}$ and $(a,b)$ is incident to $(a',b')$ if and only if $(a,a') \in E(A)$ and $b = b'$ or  $(b,b') \in E(B)$ and $a = a'$.  An example is given in Figure \ref{F:Cart}.  We begin by presenting a folkloric result.

\begin{figure}
\begin{center}
\includegraphics[scale=.5]{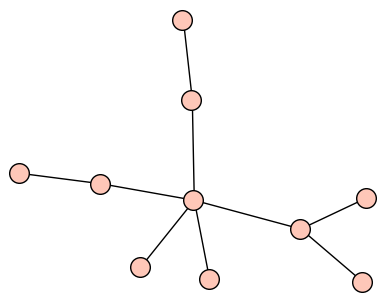}
\includegraphics[scale=.4]{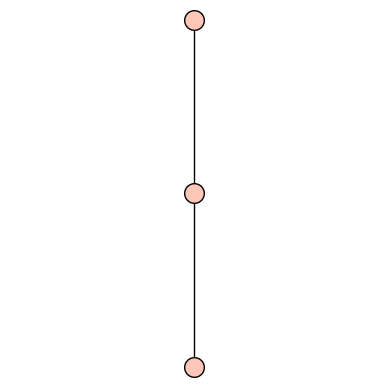}
\includegraphics[scale=.5]{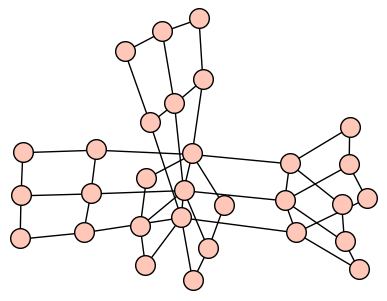}
\caption{A randomly generated BA(10,1) graph \cite{Barabasi}, $P_3$, and their Cartesian product, respectively.}
\label{F:Cart}
\end{center}
\end{figure}

\begin{lemma}
\label{L:CPEigen}
Let $(\lambda, x)$ and $(\rho, y)$ be the principal eigenvectors of $A$ and $B$, respectively.  Let $z \in \mathbb{R}^{|A||B|}$ where $z = x \otimes y$ that is $z_{ab} = x_ay_b$.  Then $(\lambda + \rho, z)$ is the principal eigenpair of $A \square B$.
\end{lemma}

\begin{proof}
We have 
\begin{align*}
(A\square B)z_{ab} &= \left(\sum_{a' \in N_A(a)} z_{a'b} + \sum_{b' \in N_B(b)} a_{ab'}\right)z_{ab} \\
&= \sum_{a' \in N_A(a)} z_{a'b}z_{ab} + \sum_{b' \in N_B(b)} z_{ab'}z_{ab} \\
&= \sum_{a' \in N_A(a)} (x_{a'}y_b)(x_ay_b) + \sum_{b' \in N_B(b)} (x_ay_{b'})(x_ay_b) \\
&= y_b\left(\sum_{a' \in N_A(a)} x_{a'}\right)x_a + x_a\left(\sum_{b' \in N_B(b)} y_{b'}\right) y_b \\
&=y_b Ax_a + x_a By_b = y_b(\lambda x_a) + x_a(\rho y_b) \\
&= (\lambda + \rho)x_ay_b = (\lambda + \rho)z_{ab}.
\end{align*}
\end{proof}

We show that $c_e(G^{\square k})$ has the following form.

\begin{lemma}
\label{L:product}
We have
\[
c_e(A \square B) = c_e(A)c_e(B) + c_e(A) + c_e(B).
\]
\end{lemma}

\begin{proof}
Let $(\lambda, x)$ and $(\rho, y)$ be the principal eigenpair of $A$ and $B$, respectively. From Lemma \ref{L:CPEigen} we have that $z = x \otimes y$ is the principal eigenvector of $A \square B$.  Note that
\[
\frac{||z||_2^2}{||z||_1^2} = \frac{\sum (x_ay_b)^2}{(\sum x_ay_b)^2} = \frac{(\sum x_a^2)(\sum y_b^2)}{(\sum x_a)^2(\sum  y_b)^2} = \frac{||x||_2^2 ||y||_2^2}{||x||_1^2||y||_1^2}.
\]

So that
\begin{align*}
c_e(A \square B) &= |V(A)||V(B)|\left(\frac{||x||_2^2 ||y||_2^2}{||x||_1^2||y||_1^2}\right) - 1 \\
&= \left(|V(A)|\frac{||x||_2^2}{||x||_1^2}\right)\left(|V(B)|\frac{||y||_2^2}{||y||_1^2}\right) - 1\\
&= (c_e(A)+1)(c_e(B)+1) - 1 \\
&= c_e(A)c_e(B) + c_e(A) + c_e(B).
\end{align*}
\end{proof}

We now consider $c_d(G^{\square k})$.  We adopt the notation $A+B = \{a + b : a \in A, b \in B\}$ to be the \textit{multiset sum}.  For convenience, we write $kA = \sum_{i=1}^k A$.  Note that $|kA| = |A|^k$.  We make use of the following facts.

\begin{lemma}
Let $A \subset \mathbb{Z}^+$ be a multiset.  Then 
\[
||kA||_1 = k|A|^{k-1}||A||_1.
\]
\end{lemma}

\begin{proof}
The case of $k=1$ is trivial.  Proceeding by induction we find
\begin{align*}
    ||kA||_1 &= \sum_{i_1}\dots \sum_{i_k}(a_{i_1} + \dots + a_{i_k}) \\
    &= \sum_{i_1} \dots \sum_{i_{k-1}}\left(|A| (a_{i_1} + \cdots + a_{k-1}) + ||A||_1\right) \\
    &= |A| \left(\sum_{i_1} \dots \sum_{i_{k-1}} a_{i_1} + \cdots + a_{k-1}\right) + |A|^{k-1}||A||_1 \\
    &= |A|((k-1)|A|^{k-2}||A||_1) + A^{k-1}||A||_1 = k|A|^{k-1}||A||_1.
\end{align*}
\end{proof}

\begin{lemma}
Let $A \subset \mathbb{Z}^+$ be a multiset.  Then
\[
||kA||_2^2 = k|A|^{k-2}(|A| \cdot ||A||_2^2 + (k-1)||A||_1^2)
\]
\end{lemma}

\begin{proof}
The case of $k = 1,2,3$ are clear.  We further have that
\begin{align*}
||kA||_2^2 &= \sum_{i_1}\dots \sum_{i_k} (a_{i_1} + \dots +a_{i_k})^2 \\
&=\sum_{i_1}\dots \sum_{i_k} \left((a_{i_1} + \dots +a_{i_{k-1}})^2 + a_{i_k}(2(a_{i_1} + \dots + a_{i_{k-1}}) + a_{i_k})\right) \\
&= \left(\sum_{i_1}\dots \sum_{i_k} (a_{i_1} + \dots +a_{i_{k-1}})^2 \right) + \left(\sum_{i_1}\dots \sum_{i_k}a_{i_k}(2(a_{i_1} + \dots + a_{i_{k-1}}) + a_{i_k})\right)\\
&= |A|\cdot ||(k-1)A||_2^2 + \left(2||A||_1 ||(k-1)A||_1 + |A|^{k-1}||A||_2^2 \right)\\
&=|A|\big((k-1)|A|^{k-3}(|A|||A||_2^2 + (k-2)||A||_1^2\big) \\
&= k|A|^{k-1}||A||_2^2 + k(k-1)|A|^{k-2}||A||_1^2\\
&= k|A|^{k-2}(|A|\cdot ||A||_2^2 + (k-1)||A||_1^2).
\end{align*}
\end{proof}

We now find a simple equation for $c_d(G^{\square k})$.
\begin{lemma}
\label{L:degcart}
For a graph $G$, 
\[
c_d(G^{\square k}) = \frac{c_d(G)}{k}.
\]
\end{lemma}
\begin{proof}
We have that
\begin{align*}
c_d(G^{\square k}) &= |kd|\frac{||kd||_2^2}{||kd||_1^2} - 1 \\
&= |d|^k \left(\frac{k|d|^{k-2}(|d|\cdot ||d||_2^2 + (k-1)||d||_1^2)}{(k|d|^{k-1}||d||_1)^2} \right)- 1 \\
&= \frac{|d|\cdot ||d||_2^2 + (k-1)||d||_1^2}{k||d||_1^2}- 1 \\
&= \frac{|d|\cdot ||d||_2^2 -||d||_1^2}{k||d||_1^2} = \frac{1}{k}\left(|d| \frac{||d||_2^2}{||d||_1^2} - 1\right)\\
&= \frac{c_d(G)}{k}.
\end{align*}
\end{proof}

\begin{theorem}
Let $G$ be a non-regular graph.  Then
\[
\lim_{k \to \infty} \gamma^2(G^{\square k})-1 = \infty, \lim_{k \to \infty} c_e(G^{\square k}) = \infty, \text{ and } \lim_{k \to \infty} c_d(G^{\square k}) = 0.
\]
\end{theorem}

\begin{proof}
Let $(\lambda,x)$ be the principal eigenpair of $G$.  By Lemma \ref{L:CPEigen},  $\gamma = x_\text{max}/x_\text{min}$ and  $\gamma(G^{\square k}) = (x_\text{max}/x_\text{min})^k$.  Further since $G$ is non-regular we have that $\gamma(G) > 1$ from which it follows 
\[
\lim_{k \to \infty} \gamma^2(G^{\square k})-1 = \infty.
\]
The remaining two limits follow from Lemmas \ref{L:product} and \ref{L:degcart}, respectively. 
\end{proof}

\section{Extremal Graphs}

We conclude with conjectures concerning the extremal graphs of $c_e$ and $c_d$.

The kite graph (aka lollipop graph) is the extremal graph for the maximum hitting time of a graph  as well as the maximum principal ratio for graphs on $n$ vertices \cite{Brightwell, Tait}. In each case, the head size (i.e., $s$ in $P_mK_s$) is a function of $n$.  The maximum hitting time occurs when $s \approx (2n+1)/3$ while maximising the principal ratio requires $s \approx n/\log n$.  We conjecture that a kite graph maximises $c_e$ when the head size is precisely $s=4$.

\begin{conjecture}
(``Four's a Crowd") The connected graph which achieves the maximum of $c_e$ on $n \geq 6$ vertices is $P_{n-3}K_4$.
\end{conjecture}

We also make the following conjecture about $c_d$.

\begin{conjecture}
The connected graph which achieves the maximum of $c_d$ on $n$ vertices is the star $K_{1,n-1}$. 
\end{conjecture}

Consider the irregularity measure
\[
\Gamma(G) := \frac{c_e(G) - c_d(G)}{\gamma^2(G)}
\]
which satisfies $-1 \leq \Gamma(G) \leq 1$.  We have evidence to suggest that $|\Gamma| < 1$.  Consider the following. 


\begin{theorem}
\[
\lim_{n \to \infty} \Gamma(K_{1,n}) = -1/4.
\]
Moreover, 
\[
\lim_{k \to \infty} \lim_{n \to \infty}\Gamma(S(k,nk)) = -1/4.
\]
\end{theorem}

\begin{proof}
Let $(\lambda,x)$ be the principal eigenpair of $K_{1,n}$.  Recall from Lemma \ref{L:star_limits} that $\gamma = \sqrt{n}$.  We have then
\[
\lim_{n \to \infty} \Gamma(K_{1,n}) = \lim_{n \to \infty} \frac{\left(\frac{2n}{n} - 1\right) - \left(\frac{n^3}{4n^2}-1\right)}{n} = -1/4.
\]
Moreover we have
\[
\lim_{k \to \infty} \lim_{n \to \infty}\Gamma(S(k,nk) = \lim_{k \to \infty} \frac{\frac{k\left(\frac{\sqrt{4k+1}-1}{k}-2\right)^2}{(\sqrt{4k+1}+1)^2} - \frac{k^3}{(2k+1)^2}}{\left(\frac{\sqrt{4k+1}+1}{2}\right)^2} = \lim_{k \to \infty} \frac{\frac{4}{4k} - \frac{k}{4}}{\frac{4k}{4}} = -1/4.
\]
\end{proof}

Based on experimental evidence we make the following conjecture. 

\begin{conjecture}
  $\Gamma(G)\geq -1/4$ and the star is extremal. 
\end{conjecture}

\section{Properties of Complete Split Graphs}


We showed that the complete split graph $S(n,m)$ has several interesting properties.  Firstly, the principal ratio, eigenvector dispersion, and degree dispersion of $S(n,nk)$ all converge to a positive number for fixed $k$.  In the case when $k = 1$, we showed that the principal ratio of $S(n,n)$ converges to $\varphi$, the golden ratio. We conclude by showing that the average clustering coefficient and transitivity diverge for $S(n,kn)$ as we take $n$ and $k$ to be arbitrarily large. To the best of our knowledge this is the third such graph family to have this property \cite{Estrada} (c.f., the friendship (aka windmill, $n$-fan) graph \cite{Bollobas2003} and the agave graph \cite{Estrada2}).  Notably, this property was observed in brokerage networks formed by illicit marketplaces which are closely approximated by complete split graphs \cite{Spagnoletti2021, Thomaz, ThomazSoon}.

 
 The \textit{Watts-Strogatz clustering coefficient} (see \cite{Watts}) of a vertex $i$ is a measure of the transitivity of local connections in a network
 \[
 C_i = \frac{2t_i}{d_i(d_i-1)}
 \]
 where $t_i$ is the number of triangles which contain $i$.  The \textit{average Watts-Strogatz clustering coefficient} of a graph $G$ is defined
 \[
 \overline{C} = \frac{1}{n} \sum_{i=1}^n C_i.
 \]
 Finally the \textit{transitivity} of a graph (see  \cite{Newman2001, Wasserman}) is defined to be 
 \[
 T = \frac{3(\# \text{ of triangles})}{ \sum_{i=1}^n \binom{d_i}{2}}.
 \]

\begin{lemma}
\label{L:clustering}
Consider $S(n,m)$.  Let $u \in K_n$ and $v \notin K_n$.  Then 
\[
C_u = \frac{(n-1)(n-2) + 2(n-1)m}{(n+m-1)(n+m-2)} \text{ and } C_v = 1
\]
so that
\[
\overline{C} = \frac{m + \frac{{\left({\left(n - 1\right)} {\left(n - 2\right)} + 2 \, {\left(n - 1\right)} m\right)} n}{{\left(n + m - 1\right)} {\left(n + m - 2\right)}}}{n + m}.
\]
Moreover,
\[
T = \frac{{\left(n - 1\right)} {\left(n - 2\right)} n + 3 \, {\left(n - 1\right)} n m}{{\left(n + m - 1\right)} {\left(n + m - 2\right)} n + {\left(n - 1\right)} n m}.
\]
\end{lemma}

\begin{proof}
We have that
\[
C_u = \frac{2t_u}{d_u(d_u-1)} = \frac{2\left(\binom{n-1}{2} + \binom{n-1}{1}\binom{m}{1}\right)}{(n+m-1)(n+m-2)} =  \frac{(n-1)(n-2) + 2(n-1)m}{(n+m-1)(n+m-2)}
\]
and 
\[
C_v = \frac{2t_v}{d_v(d_v-1)} = \frac{2\binom{n}{2}}{n(n-1)} = 1.
\]
The formula for $\overline{C}$ follows by straight forward computation.  Moreover, our equation for $T$ follows naturally from the fact that the number of triangles in $S(n,m)$ is 
\[
\binom{n}{3} + \binom{n}{2}\binom{m}{1}
\]
and 
\[
\sum_{i} \binom{d_i}{2} = n\binom{n+m-1}{2} + m\binom{n}{2}.
\]
\end{proof}


\begin{theorem}
\[
\lim_{m \to \infty} \overline{C}(S(n, m)) = 1 \text{ and } \lim_{m \to \infty} T(S(n, m)) = 0
\]
and 
\[
\lim_{k \to \infty} \lim_{n \to \infty} \overline{C}(S(n, kn)) = 1 \text{ and } \lim_{k \to \infty} \lim_{n \to \infty} T(S(n,kn)) =0.
\]
\end{theorem}

\begin{proof}
From Lemma \ref{L:clustering} we have 
\[
\lim_{m \to \infty} \overline{C}(S(n, m)) = \lim_{n \to \infty} \frac{m + \frac{2n^2m}{m^2}}{m} = 1
\]
and 
\[
\lim_{m \to \infty} T(S(n, m)) = \lim_{m \to \infty} \frac{3n^2m}{m^2} = 0.
\]
Now let $m = nk$. By substitution we find
\[
\overline{C}(S(n,nk)) = \frac{k n + \frac{{\left(2 \, k {\left(n - 1\right)} n + {\left(n - 1\right)} {\left(n - 2\right)}\right)} n}{{\left(k n + n - 1\right)} {\left(k n + n - 2\right)}}}{k n + n}.
\]
and 
\[
T(S(n,nk)) = \frac{3 \, k {\left(n - 1\right)} n^{2} + {\left(n - 1\right)} {\left(n - 2\right)} n}{k {\left(n - 1\right)} n^{2} + {\left(k n + n - 1\right)} {\left(k n + n - 2\right)} n}
\]
By straight forward computation we find
\[
\lim_{n \to \infty} \overline{C}(S(n, kn)) = \frac{k^{3} + 2 \, k^{2} + 3 \, k + 1}{k^{3} + 3 \, k^{2} + 3 \, k + 1} \text{ and } \lim_{n \to \infty} T(S(n,kn)) = \frac{3 \, k + 1}{k^{2} + 3 \, k + 1}.
\]
so that 
\[
\lim_{k \to \infty} \lim_{n \to \infty} \overline{C}(S(n, kn)) = 1 \text{ and } \lim_{k \to \infty} \lim_{n \to \infty} T(S(n,kn)) =0.
\]
as desired. 
\end{proof}

\bibliography{main}
\bibliographystyle{plain}

\end{document}